\documentclass[runningheads]{llncs}
\usepackage[T1]{fontenc}
\usepackage{graphicx}
\usepackage{amsmath, mathtools}
\usepackage{amssymb, bm}
\usepackage[misc]{ifsym}
\usepackage[table]{xcolor}
\usepackage{booktabs}
\usepackage{multirow}
\usepackage{tabularx}
\usepackage{hyperref}
\usepackage[linesnumbered,ruled,vlined]{algorithm2e}
\spnewtheorem{assumption}[case]{Assumption}{\bfseries}{\itshape}
\spnewtheorem{fact}[example]{Fact}{\bfseries}{\itshape}
\begin{document}
\title{PDHCG: A Scalable First-Order Method for Large-Scale Competitive Market Equilibrium Computation}
%
%
\author{Huikang Liu\inst{1} \and
Yicheng Huang\inst{2} \and
Hongpei Li\inst{3}\and
Dongdong Ge$^{(\textrm{\Letter})}$\inst{4} \and
Yinyu Ye\inst{5}}
\authorrunning{H. Liu et al.}
%
\institute{Shanghai Jiao Tong University, Shanghai, China. 
\email{hkl1u@sjtu.edu.cn}  \and Shanghai University of Finance and Economics, Shanghai, China. 
\email{huangyc@stu.sufe.edu.cn}
\and Shanghai University of Finance and Economics, Shanghai, China. 
\email{ishongpeili@gmail.com}
\and Shanghai Jiao Tong University, Shanghai, China. 
\email{ddge@sjtu.edu.cn} \and 
Stanford University, California, USA.  \email{yinyu-ye@stanford.edu}}
\maketitle              
\begin{abstract}
Large-scale competitive market equilibrium problems arise in a wide range of important applications, including economic decision-making and intelligent manufacturing. Traditional solution methods, such as interior-point algorithms and certain projection-based approaches, often fail to scale effectively to large problem instances. In this paper, we propose an efficient computational framework that integrates the primal-dual hybrid conjugate gradient (PDHCG) algorithm with GPU-based parallel computing to solve large-scale Fisher market equilibrium problems. By exploiting the underlying mathematical structure of the problem, we establish a theoretical guarantee of linear convergence for the proposed algorithm. Furthermore, the proposed framework can be extended to solve large-scale Arrow–Debreu market equilibrium problems through a fixed-point iteration scheme. Extensive numerical experiments conducted on GPU platforms demonstrate substantial improvements in computational efficiency, significantly expanding the practical solvable scale and applicability of market equilibrium models.

\end{abstract}

\section{Introduction}
The Arrow–Debreu competitive market equilibrium is a cornerstone of modern economic theory, tracing its origins to Walras’s foundational work in 1874~\cite{walras1900elements}. Walras posed the fundamental question of whether it is possible to assign prices to goods in such a way that every consumer maximizes their utility and all markets simultaneously clear. This question was affirmatively answered by Arrow and Debreu in 1954~\cite{arrow-debreu}, who proved the existence of such an equilibrium under the assumption of concave utility functions. However, their proof was existential in nature and did not yield a constructive method for computing the equilibrium, leaving open the algorithmic challenge of finding such prices.

Fisher was the first to consider an algorithmic approach \cite{brainard2005compute} for computing equilibrium prices in a related, yet distinct, market model where participants are divided into two sets: consumers and producers. In the Fisher market model, there are $m$ producers and $n$ consumers. Each consumer $i \in [n]$ is endowed with a budget $w_i > 0$ and seeks to purchase goods to maximize their individual utility, while each producer $j \in [m]$ offers one unit of a good for sale. A market equilibrium consists of a set of prices such that every consumer purchases an optimal bundle within their budget, and all goods are fully sold with all money fully spent—i.e., the market clears. It has been shown that if the consumers’ utility functions satisfy the CCNH conditions (Concave, Continuous, Nonnegative, and Homogeneous), then the market equilibrium can be computed by solving the Eisenberg–Gale convex program \cite{eisenberg1959consensus}:
\begin{align}\label{Eisenberg-Gale} 
    \begin{split}
        \max_{x} \quad & \sum_{i \in [n]} w_i \log u_{i} (x_{i}) \\
        \text{s.t.} \quad & \sum_{i \in [n]} x_{ij} = 1, \, \forall j \in [m] \\
        & x_{ij} \geq 0, \, \forall i,j
    \end{split}
\end{align}
where $x_{ij}$ denotes the amount of good $j$ purchased by consumer $i$, and $u_i(x_i)$ represents the total utility derived by consumer $i$ from their bundle. Several standard utility functions are commonly considered in this setting, including linear utility where $u_i(x_i) = \sum_{j \in [m]} u_{ij} x_{ij}$, CES utility where $u_i(x_i) = \left( \sum_{j \in [m]} u_{ij} x_{ij}^p \right)^{1/p}, \; p \leq 1$, and Leontief utility where $u_i(x_i) = \min_{j \in [m]} \left\{ \frac{x_{ij}}{u_{ij}} \right\}$.
In this paper, we primarily focus on the case where utility functions are linear.

The Fisher market model exhibits appealing theoretical properties, especially when all consumers have equal budgets. Under this condition, the resulting equilibrium satisfies the Competitive Equilibrium from Equal Incomes (CEEI) mechanism \cite{vazirani2007combinatorial}, which guarantees envy-freeness, Pareto optimality, and incentive compatibility in large markets \cite{budish2011combinatorial,azevedo2019strategy}. This mechanism has been successfully applied in course allocation at the Wharton School \cite{budish2016bringing}. Beyond education, Fisher markets have been used in applications such as recommendation systems, online auctions, resource allocation, and donation matching \cite{kroer2019scalable,bateni2022fair,mcelfresh2020matching,parkes2015beyond}, highlighting their broad practical relevance in promoting fairness and efficiency.

With the rapid advancement of technology and networked systems, market equilibrium problems are being formulated at increasingly larger scales. \emph{A central challenge that arises is how to efficiently compute equilibrium in large-scale Fisher market -- and more generally, in Arrow–Debreu market models.}

In 2002, Devanur et al. ~\cite{devanur2002market} proposed a combinatorial algorithm for solving the Fisher market with linear utility functions, achieving a polynomial-time complexity of $O(n^8 \log(1/\epsilon))$. Subsequently, Jain \cite{jain2007polynomial} formulated the Arrow–Debreu model as a convex inequality problem and designed an ellipsoid-based algorithm for its solution. Later, Ye \cite{ye2008path} applied interior-point methods to both Fisher and Arrow–Debreu models with linear utilities, significantly reducing the complexity to $O(n^4 \log(1/\epsilon))$. Despite these theoretical advances, the strong dependence on the problem dimension $n$ severely limits the scalability of these methods, making them impractical for solving large-scale market equilibrium problems.

Motivated by the need to address large-scale instances, first-order methods have been introduced as a promising alternative. In 2020, Gao and Kroer \cite{gao2020first} proposed a proximal gradient method for solving the Fisher market equilibrium problem and established a linear convergence rate. To further improve computational efficiency, they later incorporated block coordinate descent techniques \cite{nan2023fast}. For the Arrow–Debreu model, Devanur et al. \cite{devanur2016rational} proposed a rational convex formulation under the special setting where there is a one-to-one correspondence between participants and goods. Building on this, Chen et al. \cite{chen2022alternating} developed an alternating optimization algorithm to solve the resulting convex program. However, a major bottleneck in these approaches lies in the requirement of polyhedral projections at each iteration, which limits their scalability. As a result, the dimensionality of solvable instances remains restricted to problems of only a few thousand variables.

Recently, GPU-accelerated primal–dual first-order methods have gained considerable attention for solving large-scale convex optimization problems. These methods have demonstrated remarkable success in various domains, including linear programming \cite{applegate2021practical,lu2023cupdlp,chen2024hpr}, quadratic programming \cite{lu2023practical,huang2024restarted}, semidefinite programming \cite{han2024low,han2024accelerating}, conic programming \cite{lin2025pdcs}, optimal transport \cite{lu2024pdot,zhang2025hot}, and large-scale network flow problems \cite{zhang2025solving}; see the further related work section for more details. These GPU-enabled primal–dual methods offer significant improvements over traditional approaches in terms of both scalability and computational efficiency.
Motivated by these advances, we aim to develop an efficient GPU-accelerated primal–dual framework for solving the Fisher market model and, more generally, the Arrow–Debreu market equilibrium problem.

\subsubsection{Our Contributions.}
In this paper, we develop a scalable and efficient computational framework for solving large-scale market equilibrium problems. Our main contributions are as follows:
\begin{itemize}
    \item[1.] \emph{GPU-Accelerated Primal–Dual Framework.} We propose a primal-dual hybrid conjugate gradient (PDHCG) algorithm tailored for Fisher market models with linear utilities, and implement it on GPU architectures to exploit massive parallelism. The algorithm avoids expensive matrix factorizations, enabling efficient large-scale computation.
    \item[2.] \emph{Theoretical Guarantee.} By leveraging the underlying mathematical structure of the problem, we establish a theoretical guarantee of linear convergence for the proposed PDHCG method.
    \item[3.] \emph{Extension to Arrow–Debreu Equilibrium.} We integrate the PDHCG algorithm into a fixed-point iteration scheme, allowing it to solve large-scale Arrow–Debreu market equilibrium problems beyond the Fisher setting.
    \item[4.] \emph{Comprehensive Experimental Validation.} We conduct extensive experiments on both synthetic data and real-world transaction data (from JD.com), demonstrating that our GPU-based framework significantly outperforms traditional solvers in terms of scalability and runtime—solving problems with up to tens of millions of variables.
\end{itemize}

\subsubsection{Further Related Work.} 
First-order algorithms, which primarily rely on matrix–vector multiplications, are inherently well-suited for parallelization. This computational structure enables them to fully leverage the massive parallelism of modern GPUs, leading to significantly improved speed and scalability. In linear programming, Lu and Yang \cite{applegate2021practical} introduced the Primal-Dual Linear Programming (PDLP) method, which was later adapted to GPU platforms \cite{lu2023cupdlp}. This led to the development of cuPDLP-C \cite{lu2023cupdlpc}, the first high-performance, GPU-native first-order LP solver implemented in C. In parallel, Chen et al. \cite{chen2024hpr} proposed HPR-LP, a solver based on the Halpern–Peaceman–Rachford method with semi-proximal terms, which achieves further speedups over PDLP.

For quadratic programming, Lu \cite{lu2023practical} and Huang \cite{huang2024restarted} developed the PDQP and PDHCG algorithms, respectively, enabling the solution of problems with hundreds of millions of variables. In semidefinite programming, Han \cite{han2024low} introduced the LoRADS algorithm and its GPU-enhanced version cuLoRADS \cite{han2024accelerating}, scaling up the solvable problem size by several orders of magnitude. For general conic linear programs, Lin et al. \cite{lin2025pdcs} proposed PDCS, a large-scale primal-dual solver with GPU acceleration. In optimal transport, Lu \cite{lu2024pdot} introduced the PDOT algorithm for solving large-scale problems efficiently. These GPU-based methods have also been applied to real-world applications; for instance, Zhang \cite{zhang2025solving} leveraged GPU-accelerated solvers to address multi-commodity network flow problems at scale.

\subsubsection{Organization.} 
The remainder of this paper is organized as follows. Section~\ref{sec:algorithm} introduces the standard Primal–Dual Hybrid Gradient (PDHG) algorithm and presents our proposed PDHCG algorithm tailored for the Fisher market problem. In Section~\ref{sec:convergence}, we establish the linear convergence rate of the algorithm by proving the quadratic growth property of the smoothed duality gap and leveraging this property to derive convergence guarantees. Section~\ref{sec:arrow-debreu} extends our framework to the Arrow–Debreu market equilibrium problem by proposing a fixed-point iteration scheme that incorporates the PDHCG algorithm; we also provide convergence analysis under mild conditions. Section~\ref{sec:numerical} presents empirical evaluations on both synthetic and real-world market instances, demonstrating the efficiency and scalability of our approach. Finally, Section~\ref{sec:conclusion} concludes the paper with a summary of findings.

\section{Primal-Dual Hybrid Conjugate Gradient Method}\label{sec:algorithm}
In this section, we review the standard Primal–Dual Hybrid Gradient (PDHG) method proposed in~\cite{chambolle2011first} and introduce the PDHCG-type method~\cite{huang2024restarted} for solving the Fisher market equilibrium problem formulated as the Eisenberg–Gale convex program~\eqref{Eisenberg-Gale}. Throughout this paper, we make the following assumption:
\begin{assumption}\label{ass:positive}
    There is at least one $u_{ij} > 0$ for every $i \in [n]$, and at least one $u_{ij} > 0$ for every $j \in [m]$.
\end{assumption}
This is to say that every consumer in the market likes at least one good; and every good is valued by at least one consumer. Assumption \ref{ass:positive} is quite standard and necessary to guarantee that each good can have a positive equilibrium price \cite{ye2008path}.
\subsection{PDHG Methods}
By introducing auxiliary variables $t$, the Eisenberg–Gale model \eqref{Eisenberg-Gale} can be equivalently reformulated as:
\begin{align}\label{Eisenberg-Gale-reformulation}
    \begin{split}
        \min_{x, t} \quad & \sum_{i \in [n]} -w_i \log t_i \\
        \text{s.t.} \quad & \sum_{i \in [n]} x_{ij} = 1, \, \forall j \in [m] \\
        & t_i - \sum_{j \in [m]} u_{ij} x_{ij}=0, \, \forall i \in [n] \\
        & t_i, x_{ij} \geq 0, \, \forall i,j.
    \end{split}
\end{align}
Notice that it is computationally challenging to directly project onto the constraint set of~\eqref{Eisenberg-Gale-reformulation}, which involves both column-sum and row-sum coupling constraints. This difficulty becomes particularly pronounced in large-scale settings. To overcome this, we instead consider the corresponding primal–dual formulation of the problem:
\begin{equation}\label{eq:primal-dual-formulation}
    \min_{x\geq 0,\, t > 0} \max_{p, y} \; \mathcal{L} (x,t; p,y) = \left. \sum_{i \in [n]} - w_i \log t_i + \sum_{j \in [m]} p_j  \left(
\sum_{i \in [n]} x_{ij} - 1 \right) + \sum_{i \in [n]} y_i  \left( t_i - \sum_{j \in [m]}
u_{ij} x_{ij} \right. \right)
\end{equation}
where $p \in \mathbb{R}^m$ and $y\in \mathbb{R}^n$ are the dual variables introduced to enforce the market-clearing and utility constraints, respectively. Notably, the optimal dual variable $p^\star$ corresponds to the equilibrium prices of the goods \cite{ye2008path}. To solve the resulting saddle-point problem, we adopt the PDHG algorithm introduced by Chambolle and Pock \cite{chambolle2011first}. PDHG is a widely used first-order method for convex–concave minimax problems and achieves convergence by performing alternating updates to the primal and dual variables. The one-step update is formulated as follows:
\begin{equation}\label{eq:update-rule}
\begin{aligned}
(p^{k+1},y^{k+1})&= \mathop{\arg \max}_{p,\, y} \; \mathcal{L}(2x^{k}-x^{k-1},2t^{k}-t^{k-1}; p,y)- \frac{1}{2\sigma_k} \left(\|y-y^k\|_2^2+\|p-p^k\|_2^2 \right) \\
(x^{k+1},t^{k+1})&=\mathop{\arg \min}_{x\geq 0,\, t > 0} \; \mathcal{L}(x,t; p^{k+1}, y^{k+1}) +\frac{1}{2\tau_k} (\|x-x^k\|_2^2+\|t-t^k\|_2^2),\\
\end{aligned}
\end{equation}
where $\tau_k$ and $\sigma_k$ denote the primal and dual step-sizes. Note that the optimization problems in \eqref{eq:update-rule} are separable with respect to each variable $x_{ij}, t_i, p_j,$ and $y_i$. Especially, the maximization problems w.r.t. each $p_j$ and $y_i$ are unconstrained quadratic programs, and thus their closed-form update rules can be derived directly:
\begin{equation}\label{update:py}
    \begin{aligned}
     &p_{j}^{k+1}= p_{j}^k + {\sigma_k} \left(\sum_{i \in {N}}(2x_{ij}^{k}-x_{ij}^{k - 1})-1 \right),\\
     &y_{i}^{k+1}= y_{i}^k + {\sigma_k} \left(2t_i^{k}-t_i^{k-1}-\sum_{j \in {M}}u_{ij}(2 x_{ij}^{k}- x_{ij}^{k-1}) \right).  
\end{aligned}
\end{equation} 
While, for each $x_{ij}$ and $t_i$, the minimization problems become
\[
\begin{aligned}
    x_{ij}^{k+1} & =\arg \min_{x_{ij\geq 0}} \;p_j^{k+1} x_{ij}- u_{ij} y_i^{k+1} x_{ij} +\frac{1}{2\tau_k} (x_{ij}-x_{ij}^k)^2, \\
    t_i^{k+1} & =\mathop{\arg \min}_{t_i > 0} \; - w_i \log t_i + y_i^{k+1} t_i +\frac{1}{2\tau_k} (t_i-t_i^k)^2.
\end{aligned}
\]
The corresponding KKT condition for each $t_i$ is given by
\[
- \frac{w_i}{t_i}  + y_i^{k+1} + \frac{1}{\tau_k} (t_i - t_i^k) = 0, \quad \Leftrightarrow \quad t_i^2 + (\tau_k y_i^{k+1} - t_i^k) t_i - {\tau_k} w_i = 0.
\]
A simple computation shows that the optimal solution is give by
\begin{equation}\label{update:t}
    t_i^{k+1}=\frac{t_i^k - \tau_k y_i^{k+1} + \sqrt{(\tau_k y_i^{k+1} - t_i^k)^2 + 4 \tau_k w_i}}{2}.
\end{equation}
Similarly, for each $x_{ij}$, a simple computation shows that
\begin{align}\label{update:x}
    x_{ij}^{k+1}= \text{Proj}_{\mathbb{R}_+} (x_{ij}^k - {\tau_k} (p_j^{k+1} -  u_{ij} y_i^{k+1})).
\end{align}

Algorithm~\ref{algo:PDHG} outlines the standard two-loop restarted PDHG method. The algorithm begins with an initial point $(x^{0, 0}, t^{0, 0}, p^{0, 0}, y^{0, 0})$, a sequence of step sizes $\{(\sigma_k, \tau_k)\}$, and a specified restart interval $K$. Within each outer iteration $n$, the algorithm performs $K$ inner iterations using the standard PDHG updates~\cite{chambolle2011first}. The iterates at the $k$-th inner step are denoted by $(x^{n, k}, t^{n, k}, p^{n, k}, y^{n, k})$, and their running average is tracked as $(\overline{x}^{n, k}, \overline{t}^{n, k}, \overline{p}^{n, k}, \overline{y}^{n, k})$. After each inner loop, the algorithm is restarted, and the subsequent outer iteration is initialized using the averaged solution$(\overline{x}^{n, K}, \overline{t}^{n, K}, \overline{p}^{n, K}, \overline{y}^{n, K})$ obtained from the preceding inner loop.

The restart strategy is incorporated to enhance the convergence rate of the PDHG algorithm. While PDHG exhibits sublinear convergence for general convex–concave saddle-point problems~\cite{chambolle2016ergodic}, recent work~\cite{applegate2023faster,huang2024restarted,lu2023practical} shows that restarting can yield linear convergence in structured settings such as linear and quadratic programming, by exploiting properties like sharpness and quadratic growth. Motivated by these theoretical and practical gains, we adopt the restart scheme in our PDHG framework for the Fisher market problem. In the next section, we can see that the proposed method indeed enjoys a linear convergence guarantee.

\begin{algorithm}[t]
\DontPrintSemicolon
\SetAlgoLined
\SetInd{0.5em}{0.5em} 
\SetKwInput{KwInput}{Input}
\SetKwInput{KwOutput}{Output}
\SetKwInput{KwInitialize}{Initialize the inner loop}
\SetKwInput{KwRestart}{Restart the inner loop}
\SetKwFor{Repeat}{Repeat}{}{end} 
\SetKwFor{ForEach}{for each}{}{end} 

\KwInput{initial point $(x^{0, 0}, t^{0, 0}, p^{0,0},y^{0,0})$, step-size $\{(\sigma_k, \tau_k)\}$, restart frequency $K$.}
\Repeat{until $(x^{n, 0}, t^{n, 0},p^{n,0},y^{n, 0})$ converges}{
\KwInitialize{$(\overline{x}^{n,0},\overline{t}^{n,0},\overline{p}^{n, 0},\overline{y}^{n, 0}) \leftarrow (x^{n, 0}, t^{n, 0},p^{n,0},y^{n, 0})$}
    \ForEach{$k = 0, \ldots, K - 1$}{
        \Indp 
        Update ($p^{n, k+1}, y^{n, k+1}$) according to \eqref{update:py} \\
        Update ($x^{n, k+1},t^{n, k+1}$) according to \eqref{update:t} and \eqref{update:x} \\
        $(\overline{x}^{n, k+1},\overline{t}^{n, k+1},\overline{p}^{n, k+1},\overline{y}^{n, k+1} ) = \frac{k}{k+1}(\overline{x}^{n, k},\overline{t}^{n, k},\overline{p}^{n, k},\overline{y}^{n, k})+ \frac{1}{k+1}({x^{n, k+1},t^{n, k+1},p^{n, k+1},y^{n, k+1}})$\\
        \Indm 
    }
    \KwRestart{$({x^{n+1, 0},t^{n+1, 0},p^{n+1, 0},y^{n+1, 0}}) \leftarrow (\overline{x}^{n, K},\overline{t}^{n, K},\overline{p}^{n, K},\overline{y}^{n, K})$, $n \leftarrow n+1 $.}
}
\KwOutput{$(x^{n, 0}, t^{n, 0},p^{n,0},y^{n, 0})$}
\caption{Restarted PDHG}
\label{algo:PDHG}
\end{algorithm}

\subsection{PDHCG-Type Methods}
The standard restarted accelerated PDHG method, as presented in Algorithm~\ref{algo:PDHG}, is well-suited for GPU implementation and can scale to solve large Fisher market equilibrium problems. However, due to the introduction of auxiliary variables $t$ and the corresponding constraints $t_i - \sum_{j \in [m]} u_{ij} x_{ij} = 0$ for all $i \in [n]$, additional primal and dual variables must be updated at each iteration. This increases per-iteration complexity and may slow convergence, especially given the typically slow and circuitous convergence behavior of PDHG-type methods \cite{lu2024restarted}. To further improve convergence, we return to the original formulation of the problem:
\begin{equation}\label{pdhcg:fisher problem with eq}
    \begin{aligned}
        \max_{x} \quad & \sum_{i \in [n]} w_i \log \left( \sum_{j \in [m]} u_{ij} x_{ij} \right) \\
        \text{s.t.} \quad & \sum_{i \in [n]} x_{ij} = 1, \, \forall j \in [m] \\
        & x_{ij} \geq 0, \, \forall i,j.
    \end{aligned}
\end{equation}
We still consider its primal-dual formulation:
\begin{equation}\label{pdhcg:primal-dual-formulation}
    \min_{x\geq 0} \max_{p} \; \mathcal{L} (x, p) = \sum_{i \in [n]} - w_i \log \left( \sum_{j \in [m] } u_{ij} x_{ij} \right) + \sum_{j \in [m]} p_j  \left(
\sum_{i \in [n]} x_{ij} - 1 \right) .
\end{equation}
Although the minimax problem~\eqref{pdhcg:primal-dual-formulation} involves significantly fewer variables compared to~\eqref{eq:primal-dual-formulation}, a key challenge arises when applying the PDHG update: the subproblem with respect to $x$ is no longer separable and thus becomes difficult to solve efficiently.

The primal–dual hybrid conjugate gradient (PDHCG) method, first introduced in~\cite{huang2024restarted} for solving large-scale quadratic programming (QP) problems, is a variant of the standard PDHG algorithm. The key idea of PDHCG is to address the difficulty of solving the primal subproblem: rather than applying a single gradient descent step as in accelerated PDHG methods \cite{chen2014optimal,lu2023practical}, it employs a more efficient sub-algorithm to approximately solve the primal subproblem to a specified accuracy. This inexact update can significantly improve convergence speed and overall algorithmic efficiency. In~\cite{huang2024restarted}, the authors adopt the conjugate gradient method to solve the primal subproblem, which, in the QP case, reduces to an unconstrained quadratic optimization.

Fortunately, the subproblem with respect to $x$ remains row-separable. For each row $x_i$, the subproblem takes the form:
\begin{equation*}
    \min_{x_i \geq 0} - w_i \log \left( u_{i}^T x_{i} \right) +  (p^{k+1})^T x_{i} + \frac{1}{2\tau_k} \|x_i - x_i^k\|^2
\end{equation*}
where $u_i = (u_{ij})_j$ denote the $i$-th row of the utility matrix. The KKT conditions are given by 
\begin{equation}\label{subprob:bisection}
    \begin{aligned}
        -\frac{w_i}{u_{i}^T x_{i}} + p^{k+1} + \frac{1}{\tau_k} (x_i - x_i^k) = \lambda_i, \\
        \lambda_i^T x_i =0, \, \lambda_i \geq 0, \, x_i \geq 0.
    \end{aligned}
\end{equation}
where $\lambda_i \in \mathbb{R}^m$ is the dual variable associated with the non-negativity constraint. Define a new variable $s = u_i^Tx_i$, then, based on \eqref{subprob:bisection}, each component $x_{ij}$ can be computed as:
\begin{equation}\label{eq:bisection}
    x_{ij} = \text{proj}_{\mathbb{R}_+}\left(x_{ij}^k - \tau_k p^{k+1}_j + \frac{\tau_k w_i}{s} \right), \, \forall j \in [m].
\end{equation}  
The goal is to find a $s >0$ such that $s = \sum_{j \in [m]} u_{ij} x_{ij} = \sum_{j \in [m]} u_{ij} \cdot \text{proj}_{\mathbb{R}_+}\left(x_{ij}^k - \tau_k p^{k+1}_j + \frac{\tau_k w_i}{s} \right)$. That is, we need to solve the equation
$$
\phi(s) = s - \sum_{j \in [m]} u_{ij} \cdot \text{proj}_{\mathbb{R}_+}\left(x_{ij}^k - \tau_k p^{k+1}_j + \frac{\tau_k w_i}{s} \right) = 0.
$$
It is easy to verify that $\phi(s)$ is monotonically increasing w.r.t. $s > 0$. Therefore, the standard bisection search algorithm can be applied to efficiently identify the root of $\phi(s)$, thereby solving the KKT system~\eqref{subprob:bisection}.

\begin{algorithm}[t]
\DontPrintSemicolon
\SetAlgoLined
\SetInd{0.5em}{0.5em} 
\SetKwInput{KwInput}{Input}
\SetKwInput{KwOutput}{Output}
\SetKwInput{KwInitialize}{Initialize the inner loop}
\SetKwInput{KwRestart}{Restart the inner loop}
\SetKwFor{Repeat}{Repeat}{}{end} 
\SetKwFor{ForEach}{for each}{}{end} 

\KwInput{initial point $(x^{0, 0}, p^{0,0})$, step-size $\{(\sigma_k, \tau_k)\}$, restart frequency $K$.}
\Repeat{until $(x^{n, 0}, p^{n,0})$ converges}{
\KwInitialize{$(\overline{x}^{n,0}, \overline{p}^{n, 0}) \leftarrow (x^{n, 0},p^{n,0})$}
    \ForEach{$k = 0, \ldots, K - 1$}{
        \Indp 
        Update $p^{n, k+1}$ according to \eqref{update:py} \\
        Parallel update each row of $x^{n, k+1}$ via Algorithm \ref{algo:search} \\
        $(\overline{x}^{n, k+1},\overline{p}^{n, k+1} ) = \frac{k}{k+1}(\overline{x}^{n, k},\overline{p}^{n, k})+ \frac{1}{k+1}(x^{n, k+1},p^{n, k+1})$\\
        \Indm 
    }
    \KwRestart{$(x^{n+1, 0},p^{n+1, 0}) \leftarrow (\overline{x}^{n, K},\overline{p}^{n, K})$, $n \leftarrow n+1 $.}
}
\KwOutput{$(x^{n, 0}, p^{n,0})$}
\caption{Restarted PDHCG}
\label{algo:PDHCG}
\end{algorithm}

However, a more effective update strategy is available. Let $x_i^\star$ denote the optimal solution to~\eqref{subprob:bisection}, and define $s^\star = u_i^T x_i^\star$. Given a trial value $s = u_i^T x_i$, we can compute $\tilde{x}_{ij}$ using the update rule in~\eqref{eq:bisection}, and obtain the corresponding value $\tilde{s} = u_i^T \tilde{x}_i$. If $s < s^\star$, then $\tilde{x}_{ij} > x_{ij}^\star$, which implies $\tilde{s} > s^\star > s$. Conversely, if $s > s^\star$, then $\tilde{x}_{ij} < x_{ij}^\star$, leading to $\tilde{s} < s^\star < s$. Therefore, the values of $s$ and $\tilde{s}$ naturally form a bracketing interval that provides lower and upper bounds for $s^\star$.

The bisection search procedure is outlined in Algorithm~\ref{algo:search}. At the beginning of the algorithm, we compute an initial trial value $s$ based on the previous iterate $x_i^k$, and use it to obtain an updated value $\tilde{s}$. As established above, the pair $(s, \tilde{s})$ forms valid lower and upper bounds for the optimal value $s^\star$. At each iteration of the bisection algorithm, a new trial value is chosen as the midpoint $s = (L + U)/2$ given the current lower and upper bounds $L$ and $U$, respectively. A new value $\tilde{s}$ is then computed using the update rule. Since both $U$ and $\max\{s, \tilde{s}\}$ are valid upper bounds, we update the upper bound by taking $\min\{U, \max\{s, \tilde{s}\}\}$. Similarly, both $L$ and $\min\{s, \tilde{s}\}$ are valid lower bounds, so we update the lower bound as $\max\{L, \min\{s, \tilde{s}\}\}$.
\begin{remark}
    In the worst case, where $\tilde{s}$ lies outside the current interval $[L, U]$, the method behaves like standard bisection, halving the interval at each step. In the better case where $\tilde{s} \in [L, U]$, the bounds $(s, \tilde{s})$ form a tighter bracketing interval for $s^\star$, potentially accelerating convergence. Therefore, Algorithm~\ref{algo:search} can be viewed as an enhanced bisection method that adaptively exploits intermediate values to tighten bounds more efficiently than standard bisection.
\end{remark}

\begin{algorithm}[t]
\DontPrintSemicolon
\SetAlgoLined
\SetInd{0.5em}{0.5em} 
\SetKwInput{KwInput}{Input}
\SetKwInput{KwOutput}{Output}
\SetKwInput{KwInitialize}{Initialize the inner loop}
\SetKwInput{KwRestart}{Restart the inner loop}
\SetKwFor{Repeat}{Repeat}{}{end} 
\SetKwFor{ForEach}{for each}{}{end} 

\KwInput{parameters $x_i^k, p^{k+1}, \tau_k$, and tolerance $\varepsilon$.}
Compute a trial value $s = u_i^T x_i^k$\\
Compute $\tilde{x}_i$ according to \eqref{eq:bisection} and set $\tilde{s} = u^T_i \tilde{x}_i$ \\
Let $U = \max\{s, \tilde{s}\}$ and $L = \min\{s, \tilde{s}\}$\\
\Repeat{until $U - L \leq \varepsilon$ converges}{
Set $s = (U + L) / 2$, compute $\tilde{x}_i$ according to \eqref{eq:bisection}, and set $\tilde{s} = u^T_i \tilde{x}_i$ \\
Update $U =\min\{U, \max\{s, \tilde{s}\}\}$ and $L =\max\{L, \min\{s, \tilde{s}\}\}$
}
\KwOutput{$\tilde{x}_i$}
\caption{Bisection Search}
\label{algo:search}
\end{algorithm}

To summarize, our extensive numerical experiments in Section \ref{sec:numerical} show that the proposed PDHCG algorithm consistently requires fewer iterations than the classical PDHG method, indicating faster convergence. However, achieving high-accuracy solutions via bisection search can be time-consuming, limiting overall runtime improvement. To address this, we introduce a 32-section search algorithm tailored to GPU architecture, which significantly enhances practical efficiency and might be of independent interest.
\subsubsection{32-Section Search} 
In our GPU implementation, we adopt a 32-section search strategy in place of the standard bisection search described in Algorithm~\ref{algo:search}. Specifically, given an upper bound $U$ and a lower bound $L$, we evaluate 32 candidate points in the interval $[L, U]$, defined as: 
$$
s_l = \left(1 - \frac{l}{33}\right) L +  \frac{l}{33}U, \, l = 1, 2, \dots, 32.
$$
For each candidate $s_l$, we compute the corresponding $\tilde{x}_i^l$ in parallel according to \eqref{eq:bisection}, and then evaluate $\tilde{s}_l = u^T_i \tilde{x}_i^l$. The bounds $U$ and $L$ are then updated as follows:
\begin{equation*}
    \begin{aligned}
        U &=\min \left\{U, \max\{s_1, \tilde{s}_1\}, \max\{s_2, \tilde{s}_2\},\cdots, \max\{s_{32}, \tilde{s}_{32}\} \right\}, \\
        L &=\max\{L, \min\{s_1, \tilde{s}_1\}, \min\{s_2, \tilde{s}_2\},\cdots, \min\{s_{32}, \tilde{s}_{32}\}\}.
    \end{aligned}
\end{equation*}
The 32-section search significantly outperforms standard bisection in GPU implementations for the following three reasons:
\begin{itemize}
    \item \textbf{Warp-Level Synchronization Overhead.}
In GPU architectures, a warp consists of 32 threads that execute instructions in lockstep, meaning all threads in a warp must wait for the longest-running thread to complete its operation before the warp can proceed. When the bisection search method is executed independently by each thread within a warp, the number of iterations can vary significantly across threads, leading to inefficiencies due to thread divergence. In contrast, although the 32-section search method employs more threads overall, it exhibits lower variance in per-thread workload within a warp and thus incurs less synchronization delay from warp-level execution constraints.
\item \textbf{Memory/Compute Bandwidth Balance.}
The balance between memory access and computational workload plays a critical role in determining GPU performance. During each row update, GPU kernels must frequently access memory to retrieve required data. In the 32-section search method, all threads within a block operate on a shared set of data, enabling effective utilization of shared memory and minimizing redundant memory accesses. In contrast, the binary search method requires each thread to access distinct data during its update process, leading to substantially higher memory bandwidth consumption.
\item \textbf{Algorithmic vs. Computational Efficiency Trade-off.}
Higher-division search methods, such as 64-section or 128-section search, involve more threads per iteration but offer diminishing returns in terms of reducing the number of required iterations. As the division level increases, algorithmic efficiency degrades significantly due to the limited marginal reduction in iterations, while the computational gains become negligible. Consequently, the overall performance of such methods tends to worsen.
\end{itemize}
\begin{table}[htbp]
  \centering
  \renewcommand{\arraystretch}{1.2}
  \caption{Runtime (ms) in each iteration of different sections}
  \begin{tabular}{p{1cm}p{1cm}|cp{1cm}p{1cm}p{1cm}p{1cm}p{1cm}p{1cm}}
    \toprule
    n & m     & Sparsity \; \; & 2     & 4    & 8     & 16    & 32    & 64 \\
    \midrule
    \multirow{3}{*}{$10^5$} &  \multirow{3}{*}{4000}  & 0.5   & 247.3  & 163.2  & 131.0  & 113.3  & \textbf{95.9}  & 172.8  \\
    \cline{3-9}
    & & 0.2   & 93.6  & 67.1  & 58.0  & 50.1  & \textbf{43.3}  & 64.6  \\
    \cline{3-9}                         
    & & 0.1   & 45.4  & 30.9  & 27.2 & 23.4  & \textbf{19.4}  & 34.1  \\
    \midrule
    \multirow{3}{*}{$10^6$} & \multirow{3}{*}{4000} & 0.05  & 188.0  & 130.4  & 110.5  & 99.0  & \textbf{84.3}  & 155.9  \\
    \cline{3-9}
     & & 0.02  & 74.4  & 50.6  & 43.9  & 39.6  & \textbf{36.1}  & 68.8  \\
     \cline{3-9}
    & & 0.01  & 42.7  & 34.1  & 31.1  & 27.6  & \textbf{26.3}  & 39.8  \\
    \bottomrule
  \end{tabular}
  \label{diff_division}
\end{table}

Taking these trade-offs into account, we adopt the 32-section search method in our implementation, as it offers a balanced compromise between algorithmic and computational efficiency across a wide range of problem settings. Table~\ref{diff_division} reports the average iteration runtime of various section search algorithms applied to the Fisher market model with synthetic data (see Section~\ref{sec:numerical} for details). As shown, the 32-section search consistently yields the lowest total runtime across all settings, demonstrating its practical advantage over other division strategies.

\section{Convergence Analysis}\label{sec:convergence}
In this section, we present the computational guarantees for the PDHCG algorithm, as described in Algorithm~\ref{algo:PDHCG}. It is important to note that the analysis and results developed here are directly applicable to the PDHG algorithm in Algorithm~\ref{algo:PDHG} as well. We begin by introducing several progress metrics that are fundamental to the subsequent analysis. Let $\mathcal{Z} = \mathbb{R}_+^{n \times m} \times \mathbb{R}^m$, and define $z = (x, p) \in \mathcal{Z}$ as the pair of primal and dual variables. Furthermore, let $\mathcal{Z}^\star \subseteq \mathcal{Z}$ denote the set of optimal saddle points of the minimax problem~\eqref{pdhcg:primal-dual-formulation}.
\begin{definition}
    For any $z =(x, p)$ and $ \hat{z} =(\hat{x}, \hat{p})$, we define 
    \begin{align}
        Q(z, \hat{z}) \coloneqq \mathcal{L}(x, \hat{p}) - \mathcal{L}(\hat{x}, p).
    \end{align}
    We call $\max_{\hat{z} \in \mathcal{Z}} Q(z, \hat{z})$ the duality gap at $z$.
\end{definition}
The duality gap is a natural and widely used progress metric for analyzing primal–dual algorithms, and has appeared extensively in the literature. However, as noted in~\cite{lu2023practical}, it often suffers from divergence issues due to the potentially unbounded nature of the constraint set. To address this, a smoothed duality gap was proposed in~\cite{fercoq2022quadratic}, which mitigates this issue by introducing a penalization term based on the distance between $\hat{z}$ and a reference point $\dot{z}$. This smoothed metric is defined as follows:
\begin{definition}\label{defi:smoothedDG}
For any $\xi >0$ and $z, \dot{z} \in \mathcal{Z}$, we define 
    the smoothed duality gap at $z$ centered at $\dot{z}$ as
    \begin{align}
        G_\xi(z; \dot{z}) = \max_{\hat{z}\in \mathcal{Z}} \left\{ Q(z, \hat{z}) - \frac{\xi}{2}\|\hat{z} - \dot{z}\|^2 \right\}.
    \end{align}    
\end{definition}
Furthermore,~\cite{fercoq2022quadratic} establishes a quadratic growth condition for the smoothed duality gap, which plays a central role in their analysis of the linear convergence of the PDHG algorithm. In this work, we adopt this concept to analyze the convergence behavior of Algorithm~\ref{algo:PDHCG}:
\begin{definition}\label{def:SQG}
    Suppose $\xi > 0$. We say that the smoothed duality gap satisfies quadratic growth on a bounded set $S$ if there exists some $\alpha_{\xi} > 0$ such that, it holds for all $z^\star \in \mathcal{Z}^\star$ and any $z \in S$ that
    \begin{equation}\label{ineq:quadratic-growth}
        G_\xi(z; z^\star) \geq \alpha_{\xi} \textrm{dist}(z, \mathcal{Z}^\star)^2. 
    \end{equation}
\end{definition}
It turns out that the smoothed duality gap associated with both problems \eqref{eq:primal-dual-formulation} and \eqref{pdhcg:primal-dual-formulation} satisfies the quadratic growth property on any bounded set. A formal characterization of the corresponding growth parameter $\alpha_{\xi}$ is provided in Section \ref{subsec:QG}.
\subsection{Convergence Analysis}
We begin by presenting a decay lemma for the primal-dual hybrid gradient method, originally established in~\cite{chambolle2011first}. Part (1) of the lemma shows that the iterates remain bounded, while part (2) provides a telescoping bound on the duality gap $Q(\overline{z}^{n, K}, \hat{z})$ for each iteration. Throughout this section, we focus on the constant stepsize $\tau_k = \tau$ and $\sigma_k = \sigma$, and define $L = \max\{\|x^T \bm{1}_n\|_2 : \|x\|_F\leq 1\} = \sqrt{n}$. .

\begin{lemma}[{\cite[Theorem 1]{chambolle2011first}}] \label{lemma:decay} For any $n \geq 0$, let $z^{n, k} = (x^{n, k}, p^{n, k}), k = 1, \cdots, K$ be the sequence generated by Algorithm \ref{algo:PDHCG} in $n$-th outer iteration. Assume that the stepsize $\sigma\tau L^2 < 1$, then we have
\begin{itemize}
    \item[(1)] $\{(x^{n, k}, p^{n, k})\}$ remains bounded, indeed, for any $k$ and any $z^\star \in \mathcal{Z}^\star$:
    \begin{align}
        \frac{\|x^{n, k} - x^\star\|^2}{2\tau} + \frac{\|p^{n, k} - p^\star\|^2}{2\sigma} \leq \frac{1}{1 - \sigma\tau L^2} \left( \frac{\|x^{n, 0} - x^\star\|^2}{2\tau} + \frac{\|p^{n, 0} - p^\star\|^2}{2\sigma} \right),
    \end{align}
    \item[(2)] for the average $(\overline{x}^{n,K}, \overline{p}^{n,K})  = \left(\frac{1}{K}\sum_{k=1}^K x^{n, k}, \frac{1}{K}\sum_{k=1}^K p^{n, k}\right)$ and any $\hat{z} \in \mathcal{Z}$:
    \begin{align*}
        Q(\overline{z}^{n, K}, \hat{z}) \leq \frac{1}{K} \left( \frac{\|x^{n, 0} - \hat{x}\|^2}{2\tau} + \frac{\|p^{n, 0} - \hat{p}\|^2}{2\sigma} \right).
    \end{align*}
\end{itemize}
\end{lemma}
\begin{remark}
    To ensure the validity of the analysis, we assume that the subproblem with respect to $x$ is solved exactly so that the KKT conditions in~\eqref{subprob:bisection} are satisfied. In practice, this assumption is well justified: owing to the linear convergence of the bisection (or 32-section) search and its highly efficient GPU implementation, the subproblem can be solved to very high precision (e.g., $10^{-8}$) without introducing noticeable error or affecting overall convergence performance.
\end{remark}
The next lemma establishes the sublinear convergence rate $O(1/ K)$ of the PDHCG inner iterates with respect to the primal-dual
gap $G_\xi(\overline{z}^{n, K}; \dot{z})$. 
\begin{lemma}\label{lem:SQG:decay}
    By choosing $\sigma = \tau = \frac{1}{2L}$ and $K \geq 4L / \xi$, we have
    \begin{itemize}
        \item[(1)] for any $z^\star \in \mathcal{Z}^\star$, 
    \begin{equation}\label{ineq:dist-upperbound}
        \|\overline{z}^{n, K} - z^\star\| \leq 2 \|z^{n, 0} - z^\star\|;
    \end{equation}
    \item[(2)] for any $\dot{z} \in \mathcal{Z}$,
    \begin{align}
        G_\xi(\overline{z}^{n, K}; \dot{z}) \leq \frac{2L}{K} \|z^{n, 0} - \dot{z}\|^2.
    \end{align}
    \end{itemize}
\end{lemma}
\begin{proof}
According to the choices of stepsize $\sigma = \tau = \frac{1}{2L}$, part (1) of Lemma \ref{lemma:decay} implies that
$$
\|z^{n,k} - z^\star\|^2 \leq 4 \|z^{n,0} - z^\star\|^2
$$
holds for any $k\in [K]$ and any $z^\star \in \mathcal{Z}^\star$. Then, we have
$$
\|\overline{z}^{n,K} - z^\star\| = \left\|\frac{1}{K}\sum_{k \in [K]} z^{n,k} - z^\star \right\| \leq \frac{1}{K}\sum_{k \in [K]}\|z^{n,k} - z^\star\| \leq 2\|z^{n,0} - z^\star\|,
$$
where the first inequality holds because of the triangle inequality. So we complete the proof of part (1). Next, according to part (2) of Lemma \ref{lemma:decay}, for any $\hat{z}, \dot{z} \in \mathcal{Z}$, we know that
    \begin{align}
        Q(\overline{z}^{n, K}, \hat{z}) - \frac{\xi}{2} \|\hat{z} - \dot{z}\|^2 \leq \frac{L}{K} \|z^{n, 0} - \hat{z}\|^2 - \frac{2L}{K} \|\hat{z} - \dot{z}\|^2 \leq \frac{2L}{K} \|z^{n, 0} - \dot{z}\|^2
    \end{align}
    where the first inequality holds because we choose $K \geq 4L / \xi$ so that $\xi/2 \geq 2L/ K$, and the last inequality holds because of Cauchy–Schwarz inequality. Therefore, we have
    \begin{align*}
        G_\xi(\overline{z}^{n, K}; \dot{z}) = \max_{\hat{z}} \left\{ Q(\overline{z}^{n, K}, \hat{z}) - \frac{\xi}{2} \|\hat{z} - \dot{z}\|^2 \right\} \leq \frac{2L}{K} \|z^{n, 0} - \dot{z}\|^2.
    \end{align*}
\end{proof}
The following theorem presents our major theoretical result of Algorithm \ref{algo:PDHCG} for Fisher market model.
\begin{theorem}
     By choosing $\sigma = \tau = \frac{1}{2L}$ and $K \geq \max\left\{ \frac{4L}{\xi}, \frac{4L}{\alpha_\xi}\right\}$, we have for all $n \geq 0$ that
     \begin{align}\label{ineq:converge}
     \textrm{dist}(z^{n+1, 0}, \mathcal{Z}^\star) \leq \frac{1}{2} \textrm{dist}(z^{n, 0}, \mathcal{Z}^\star).
 \end{align}
\end{theorem}
\begin{proof}
First of all, suppose that the sequence $\{z^{n, 0}\}$ always lies in a bounded set $S$, then we can apply the quadratic growth property stated in Definition \ref{def:SQG}. Let $z^\star = \text{proj}_{\mathcal{Z}^\star} (z^{n, 0})$, we have
\begin{align*}
    \text{dist}(z^{n+1, 0}, \mathcal{Z}^\star)^2 \leq & \frac{1}{\alpha_{\xi}} G_\xi(z^{n+1,0}; z^\star) = \frac{1}{\alpha_{\xi}} G_\xi(\overline{z}^{n,K}; z^\star) \\
    \leq & \frac{2L}{K \alpha_{\xi}} \|z^{n, 0} - z^\star \|^2 = \frac{2L}{K \alpha_{\xi}} \text{dist}(z^{n, 0}, \mathcal{Z}^\star)^2 \\
    \leq & \frac{1}{2} \text{dist}(z^{n, 0}, \mathcal{Z}^\star)
\end{align*}
where the first inequality follows from the quadratic growth property in Definition \ref{def:SQG}, the second inequality follows from Lemma \ref{lem:SQG:decay}, and the last inequality holds by our choice $K \geq 4L/\alpha_\xi$. So we complete the proof.
The remaining part is to verify the assumption that the sequence $\{z^{n, 0}\}$ remains bounded. We can show by induction that \begin{equation}\label{dist:sequence-bounded}
    \text{dist}(z^{n, 0}, \mathcal{Z}^\star) \leq 2\textrm{dist}(z^{0, 0}, \mathcal{Z}^\star), \, \forall  n \geq 0.
\end{equation}
Then, the sequence is bounded since $\mathcal{Z}^\star$ is bounded. Firstly, \eqref{dist:sequence-bounded} trivally holds for $n = 0$. Assume \eqref{dist:sequence-bounded} holds for all $n \leq N$, then \eqref{ineq:converge} also holds for all $n \leq N - 1$, which implies that $\text{dist}(z^{N, 0}, \mathcal{Z}^\star) \leq \textrm{dist}(z^{0, 0}, \mathcal{Z}^\star) / 2^{N} \leq \textrm{dist}(z^{0, 0}, \mathcal{Z}^\star)$. Then, according to \eqref{ineq:dist-upperbound} in Lemma \ref{lem:SQG:decay} with $z^\star = \text{proj}_{\mathcal{Z}^\star} (z^{N, 0})$, we have
$$\text{dist}(z^{N + 1, 0}, \mathcal{Z}^\star) \leq \| z^{N + 1, 0}- z^\star\| \leq 2 \| z^{N , 0}- z^\star\| =  2\textrm{dist}(z^{N, 0}, \mathcal{Z}^\star) \leq 2\textrm{dist}(z^{0, 0}, \mathcal{Z}^\star).
$$
By induction, we complete the proof.
\end{proof}

\subsection{Quadratic Growth}\label{subsec:QG}
The quadratic growth property~\eqref{ineq:quadratic-growth} of the smoothed duality gap is generally difficult to establish and has only been verified for a few structured problems—such as quadratic programming~\cite{lu2023practical}—typically through case-by-case analysis. Fortunately, recent work by~\cite{liu2025} demonstrates that, for general convex and smooth optimization problems with linear constraints, the quadratic growth property can be implied by an appropriate KKT error bound--also known as the metric subregularity of the KKT operator \cite{alacaoglu2022convergence,drusvyatskiy2018error,liang2016convergence}--for the original problem. Therefore, to establish the quadratic growth property for our setting, it suffices to analyze the KKT error bound for the Fisher market problem. 

\subsubsection{KKT Error Bound for PDHG} 
We need to first establish the KKT error bound for the PDHG reformulation~\eqref{Eisenberg-Gale-reformulation}, and then extend the result to derive the corresponding KKT error bound for the PDHCG formulation~\eqref{pdhcg:primal-dual-formulation}. The KKT conditions for the Fisher equilibrium problem~\eqref{Eisenberg-Gale-reformulation} are given by:
\begin{subequations}\label{eq:KKT-system}
\begin{align}
    t_i y_i - w_i = & 0, \;\;\forall i\in [n] \label{KKT-system-1}\\
    x_{ij}(p_j - u_{ij} y_i) = & 0, \;\; \forall i, j \label{KKT-system-2}\\
    p_j - u_{ij} y_i \geq & 0, \;\; \forall i, j\label{KKT-system-3}\\
    \sum_{i \in [n]} x_{ij} = & 1, \;\; \forall j \in [m] \label{KKT-system-4}\\
    t_i - \sum_{j \in [m]} u_{ij} x_{ij} = & 0, \;\; \forall i\in [n]\label{KKT-system-5}\\
    x_{ij} \geq & 0, \;\; \forall i, j \label{KKT-system-6}.
\end{align}
\end{subequations}
It is straightforward to verify that the complementary slackness condition \eqref{KKT-system-2} is equivalent to the following formulation \cite{lu2023practical}, which avoids the explicit use of quadratic terms:
\begin{equation}\label{eq:equivalent-form}
    x_{ij} - [x_{ij} - \eta(p_j - u_{ij} y_i)]_+ = 0, \text{ for any } \eta >0.
\end{equation}
This equivalence holds in the following way: if \eqref{KKT-system-2} holds, then the condition in \eqref{eq:equivalent-form} is satisfied for any $\eta > 0$; conversely, if \eqref{eq:equivalent-form} holds for some $\eta > 0$, then \eqref{KKT-system-2} must also hold. Therefore, we can replace \eqref{KKT-system-2} with \eqref{eq:equivalent-form} for some $\eta > 0$ and define the following scaled KKT residual \cite{lu2023practical}.
\begin{definition}
    For any $\xi> 0$ and any point $(x, t, p, y)$ with $x\geq 0$, we call
    \begin{equation}
        F_{\xi}(x, t, p, y) = \begin{pmatrix}
            (t_iy_i - w_i)_i \\
            \left(x_{ij} - [x_{ij} - \frac{1}{\xi}(p_j - u_{ij} y_i)]_+\right)_{ij}\\
            ([p_j - u_{ij} y_i]_-)_j \\
            (\sum_{i \in [n]} x_{ij} - 1)_j\\
    (t_i - \sum_{j \in [m]} u_{ij} x_{ij})_i
        \end{pmatrix}
    \end{equation}
the scaled KKT residual of \eqref{Eisenberg-Gale-reformulation} at $(x, t, p, y)$.
\end{definition}
The scaled KKT residual $F_{\xi}(x, t, p, y)$ serves as a proxy for assessing the optimality of a primal-dual solution $(x, t, p, y)$. Specifically, $(x, t, p, y)$ is an optimal solution to \eqref{Eisenberg-Gale-reformulation} if and only if $F_{\xi}(x, t, p, y) = 0$. Furthermore, we establish the following KKT error bound condition which quantitatively links the distance to optimality with the residual norm in any bounded set. 
\begin{proposition}\label{prop:KKT-error}
    Let $\mathcal{X}^\star$ denote the set of solutions to the KKT system \eqref{eq:KKT-system}. For any $R> 0$, there exists some $\gamma_{\xi} >0$ such that
    \begin{equation}
        \|F_{\xi}(x, t, p, y)\| \geq \gamma_{\xi} \cdot \text{dist}((x, t, p, y), \mathcal{X}^\star), \quad \forall \|(x, t, p, y)\| \leq R.
    \end{equation}
\end{proposition}
Before proceeding, we need to introduce some facts about the KKT system \eqref{eq:KKT-system}.
\begin{fact}\label{fact}
    Under Assumption \ref{ass:positive}, we have 
    \begin{itemize}
        \item[(1)] the optimal solution $t^\star$ is unique;
        \item[(2)] the equilibrium price $p^\star$ is unique and lies in the simplex $\{p \in \mathbb{R}^m_+ \mid \sum_{j \in [m]} p_j = \sum_{i \in [n]} w_i\}$;
        \item[(3)] $\mathcal{X}^\star$ is a bounded polyhedron.
    \end{itemize}
\end{fact}
\begin{proof}
    For part (1), we prove this by contradiction. Suppose there exists two distinct optimal solutions $(t^\star, x^\star)$ and $(\tilde{t}^\star, \tilde{x}^\star)$ such that $t^\star \neq \tilde{t}^\star$. Since the feasible set is convex, their average $(\frac{t^\star + \tilde{t}^\star}{2}, \frac{x^\star + \tilde{x}^\star}{2})$ must also be a feasible solution. However, the objective function $ f(t) = \sum_{i \in [n]} -w_i \log t_i$ is strictly convex, which implies that  
    $$
    f\left(\frac{t^\star + \tilde{t}^\star}{2}\right) < \frac{f(t^\star) + f(\tilde{t}^\star)}{2}. 
    $$
    This contradicts the fact that both $\tilde{t}^\star$ and $\tilde{t}^\star$ are optimal, since their average would yield a strictly smaller objective value. Therefore, the optimal solution $t^\star$ must be unique. 

    For part (2), the uniqueness is a direct result of \cite[Corollary 2]{ye2008path}. Besides, based on the KKT system \eqref{eq:KKT-system}, we know that
\begin{equation}\label{eq:sum-pj}
    \sum_{j \in [m]} p^\star_j = \sum_{j \in [m]}  \sum_{i \in [n]} p^\star_j x_{ij} = \sum_{j \in [m]}  \sum_{i \in [n]} u_{ij}y_i x_{ij} = \sum_{i \in [n]} y_i t_i = \sum_{i \in [n]} w_i
\end{equation}
where the first, second, third, and the last equality hold because of \eqref{KKT-system-4}, \eqref{KKT-system-2}, \eqref{KKT-system-5}, and \eqref{KKT-system-1}, respectively. While the non-negativity of $p^\star$ comes from \eqref{KKT-system-3}.

For part (3), since $t^\star$, $y^\star$ and $p^\star$ are unique, the system \eqref{eq:KKT-system} implies that $\mathcal{X}^\star$ is a polyhedron w.r.t. $x$ and thus bounded.
\end{proof}
Now, we are ready to prove Proposition \ref{prop:KKT-error}. 
\begin{proof}[Proof of Proposition \ref{prop:KKT-error}]
    We first consider the following modified problem of \eqref{Eisenberg-Gale-reformulation}, where the objective function is replaced by the second-order Taylor expansion of $ f(t) = \sum_{i \in [n]} -w_i \log t_i$ at $t^\star$:
\begin{align}\label{Equivalent-form}
    \begin{split}
        \min_{x, t} \quad &  \sum_{i \in [n]} -w_i \log t_i^\star - \frac{w_i}{t_i^\star}(t_i - t_i^\star) + \frac{w_i}{2t_i^{\star 2}}(t_i - t_i^\star)^2  \\
        \text{s.t.} \quad & \sum_{i \in [n]} x_{ij} = 1, \, \forall j \in [m] \\
        & t_i - \sum_{j \in [m]} u_{ij} x_{ij}=0, \, \forall i \in [n] \\
        & t_i, x_{ij} \geq 0, \, \forall i,j.
    \end{split}
\end{align}
The KKT system to the problem \eqref{Equivalent-form} is same as \eqref{eq:KKT-system} with replacing $t_iy_i = w_i$ by 
\begin{equation*}
    -\frac{w_i}{t_i^\star} +\frac{w_i}{t_i^{\star 2}}(t_i - t^\star_i) + y_i = 0, \; \Leftrightarrow \; y_i^\star(t_i - t^\star_i) + t_i^\star(y_i - y^\star_i) = 0
\end{equation*}
where $y^\star_i = w_i / t^\star_i $ denote the optimal dual solution according to \eqref{KKT-system-1}. It is easy to see that $t^\star$ is also a optimal solution to the modified problem \eqref{Equivalent-form} by checking the KKT conditions, and is also unique since the objective is strongly convex regarding $t$. Thus, we know that $\mathcal{X}^\star$ is also the set of solution to the KKT system of  \eqref{Equivalent-form}. Then, we can define the modified KKT residual as
\begin{equation}
    \tilde{F}_{\xi}(x, t, p, y) = \begin{pmatrix}
            (y_i^\star(t_i - t^\star_i) + t_i^\star(y_i - y^\star_i))_i \\
            \left(x_{ij} - [x_{ij} - \eta(p_j - u_{ij} y_i)]_+\right)_{ij}\\
            ([p_j - u_{ij} y_i]_-)_{ij} \\
            (\sum_{i \in [n]} x_{ij} - 1)_j\\
    (t_i - \sum_{j \in [m]} u_{ij} x_{ij})_i
        \end{pmatrix}.
\end{equation}
Notice that every entry in $\tilde{F}_{\xi}(x, t, p, y)$ is a piecewise linear function in $(x, t, p, y)$, thus the $l_2$ norm of $\tilde{F}_{\xi}(x, t, p, y)$ is a sharp function, name, there exists some $\beta_{\xi} >0$ such that it holds for any $(x, t, p, y)$ within the bounded set $S = \{(x, t, p, y): \|(x, t, p, y)\| \leq R\}$ that
\begin{equation}\label{ineq:kkt-tilde}
        \|\tilde{F}_{\xi}(x, t, p, y)\| \geq \beta_{\xi} \cdot \text{dist}((x, t, p, y), \mathcal{X}^\star).
    \end{equation}
Finally, it is easy to see that
$$
y_i^\star(t_i - t^\star_i) + t_i^\star(y_i - y^\star_i) = t_i y_i - w_i - (t_i - t^\star_i)(y_i - y^\star_i)
$$
which implies that
\begin{equation*}
    \begin{aligned}
        \left| \|y_i^\star(t_i - t^\star_i) + t_i^\star(y_i - y^\star_i)\| - \|t_i y_i - w_i\| \right| \leq   \|(t_i - t^\star_i)(y_i - y^\star_i)\| 
        \leq \frac{1}{2}(\|t_i - t^\star_i\|^2 + \|y_i - y^\star_i\|^2) 
        \leq \frac{1}{2} \text{dist}((x, t, p, y), \mathcal{X}^\star)^2.
    \end{aligned}
\end{equation*}
Using the fact that $\sqrt{a^2 + b^2} - \sqrt{a'^2 + b^2} \leq |a - a'|$, we have
$$
\|\tilde{F}_{\xi}(x, t, p, y)\| - \|F_{\xi}(x, t, p, y)\| \leq |\|y_i^\star(t_i - t^\star_i) + t_i^\star(y_i - y^\star_i)\| - \|t_i y_i - w_i\| |\leq \frac{1}{2} \text{dist}((x, t, p, y), \mathcal{X}^\star)^2.
$$
Combining this with \eqref{ineq:kkt-tilde}, we get
\begin{equation}\label{ineq:dist:beta}
        \|F_{\xi}(x, t, p, y)\| \geq \beta_{\xi} \text{dist}((x, t, p, y), \mathcal{X}^\star) - \frac{1}{2}\text{dist}((x, t, p, y), \mathcal{X}^\star)^2 \geq \frac{\beta_{\xi}}{2} \text{dist}((x, t, p, y), \mathcal{X}^\star)
\end{equation}
where the last inequality hold if $\text{dist}((x, t, p, y), \mathcal{X}^\star) < \beta_{\xi}$. Besides, consider the complementary region  $\overline{S} = \{(x, t, p, y) \in S \mid \text{dist}((x, t, p, y), \mathcal{X}^\star) \geq \beta_{\xi}\}$, since $\overline{S}$ is compact and $\|F_{\xi}(x, t, p, y)\|$ is strictly positive over the region, there exists a constant $\vartheta >0$ such that
\begin{equation}\label{ineq:dist:theta}
    \|F_{\xi}(x, t, p, y)\| \geq \vartheta \geq \frac{\vartheta}{R + \|\mathcal{X}^\star\|} \text{dist}((x, t, p, y), \mathcal{X}^\star), \quad \forall (x, t, p, y) \in \overline{S},
\end{equation}
where $\|\mathcal{X}^\star\|$ denote the largest norm of the optimal solution, which is finite since $\mathcal{X}^\star$ is bounded. Combining inequalities \eqref{ineq:dist:beta} and \eqref{ineq:dist:theta}, we conclude that for any $(x, t, p, y)$ with $\|(x, t, p, y)\| \leq R$, the following holds:
$$
\|F_{\xi}(x, t, p, y)\| \geq \min\left\{  \frac{\vartheta}{R + \|\mathcal{X}^\star\|}, \frac{\beta_{\xi}}{2} \right\} \text{dist}((x, t, p, y), \mathcal{X}^\star).
$$
This completes the proof.
\end{proof}

\subsubsection{KKT Error Bound for PDHCG} 
The KKT conditions for the Fisher equilibrium problem~\eqref{pdhcg:primal-dual-formulation} are given as follows:
\begin{subequations}\label{eq:KKT-system-pdhcg}
\begin{align}
    x_{ij} \left(p_j - \frac{u_{ij} w_i}{\sum_{j \in [m]} u_{ij} x_{ij}} \right) = & 0, \;\; \forall i, j \label{KKT-system-6}\\
    p_j - \frac{u_{ij} w_i}{\sum_{j \in [m]} u_{ij} x_{ij}} \geq & 0, \;\; \forall i, j\label{KKT-system-7}\\
    \sum_{i \in [n]} x_{ij} = & 1, \;\; \forall j \in [m] \label{KKT-system-8}\\
    x_{ij} \geq & 0, \;\; \forall i, j \label{KKT-system-9}.
\end{align}
\end{subequations}
This system is equivalent to \eqref{eq:KKT-system}, with $t_i$ and $y_i$ substituted using the relations $t_i = \sum_{j \in [m]} u_{ij} x_{ij}$ and $y_i = w_i / t_i$. Similar to \eqref{eq:equivalent-form}, we can replace the complementary slackness condition \eqref{KKT-system-6} with the following condition to avoid explicit use of quadratic terms:
\begin{equation}\label{eq:equivalent-form-pdhcg}
    x_{ij} - \left[x_{ij} - \eta \left(p_j - \frac{u_{ij} w_i}{\sum_{j \in [m]} u_{ij} x_{ij}}\right)\right]_+ = 0, \text{ for any } \eta >0.
\end{equation}
Then, we define the following scaled KKT residual of \eqref{pdhcg:primal-dual-formulation}.
\begin{definition}
    For any $\xi> 0$ and any point $(x, t, p, y)$ with $x\geq 0$, we call
    \begin{equation}
        \mathcal{F}_{\xi}(x, p) = \begin{pmatrix}
            \left(x_{ij} - \left[x_{ij} - \eta \left(p_j - \frac{u_{ij} w_i}{\sum_{j \in [m]} u_{ij} x_{ij}}\right)\right]_+\right)_{ij}\\
            \left(\left[p_j - \frac{u_{ij} w_i}{\sum_{j \in [m]} u_{ij} x_{ij}} \right]_-\right)_j \\
            (\sum_{i \in [n]} x_{ij} - 1)_j
        \end{pmatrix}
    \end{equation}
the scaled KKT residual of \eqref{pdhcg:primal-dual-formulation} at $(x, p)$.
\end{definition}
Note that, for any $(x, p) \in \mathcal{Z}$, by choosing $t_i = \sum_{j \in [m]} u_{ij} x_{ij}$ and $y_i = \frac{w_i}{t_i}$ for each $i \in [n]$, it is easy to see that
$$
\mathcal{F}_{\xi}(x, p) = F_{\xi}(x, t, p, y).
$$
Then, letting $(x^\star, t^\star, p^\star, y^\star) = \text{Proj}_{\mathcal{X}^\star}(x,t, p, y)$, we have
\begin{equation}\label{ineq:KKT-error-pdhcg}
    G_{\xi}(x, p) = F_{\xi}(x, t, p, y) \geq \gamma_{\xi} \text{dist}((x, t, p, y),\mathcal{X}^\star )^2 \geq \gamma_{\xi} \|(x, p) - (x^\star, p^\star)\|^2 \geq \gamma_{\xi}\text{dist}((x, p), \mathcal{Z}^\star)^2,
\end{equation}
where the first inequality holds because of Proposition \eqref{prop:KKT-error}, the second one holds because of the fact $(x^\star, t^\star, p^\star, y^\star) = \text{Proj}_{\mathcal{X}^\star}(x,t, p, y)$, and the last inequality holds because $(x^\star, p^\star) \in \mathcal{Z}^\star$. Thus, \eqref{ineq:KKT-error-pdhcg} shows that the same KKT error bound also holds for the Fisher equilibrium problem~\eqref{pdhcg:primal-dual-formulation}.

Finally, the quadratic growth property for both the problem~\eqref{eq:primal-dual-formulation} and the problem~\eqref{pdhcg:primal-dual-formulation} follows directly from Proposition~\ref{prop:KKT-error}, \eqref{ineq:KKT-error-pdhcg}, and Theorem 1 in~\cite{liu2025}. This leads to the following corollary:
\begin{corollary}
For both minimax problems~\eqref{eq:primal-dual-formulation} and~\eqref{pdhcg:primal-dual-formulation}, the quadratic growth condition defined in~\eqref{ineq:quadratic-growth} holds with some constant $\alpha_\xi > 0$.
\end{corollary}

\section{Extension to Arrow-Debreu Exchange Market}\label{sec:arrow-debreu}
In the Arrow–Debreu market model, there are $n$ players and $m$ divisible goods. Each player $i$ is initially endowed with a fraction $e_{ij}\geq 0$ of good $j$. Without loss of generality, the total endowment of each good is normalized to one, i.e., $\sum_{i \in [n]} e_{ij} = 1$ for all $j \in [m]$. Each player seeks to purchase a bundle of goods that maximizes their utility, using the revenue obtained from selling their initial endowment. A price equilibrium is an assignment of prices to goods such that, when each player purchases their utility-maximizing bundle, all markets clear—meaning that the total demand for each good equals its total supply.

The key distinction between Fisher’s and Walras’ (Arrow–Debreu) models lies in the treatment of initial endowments. In the Arrow–Debreu model, each player acts as both a producer and a consumer, and their initial endowment is not fixed in advance; instead, it is determined by the equilibrium prices. As noted in \cite{ye2008path}, computing an Arrow–Debreu equilibrium reduces to finding an initial budget vector $w \in \mathbb{R}^n$ that satisfies the fixed-point condition:
\begin{equation}\label{eq:fixed-point}
w = E \cdot p(w),
\end{equation}
where $E = (e_{ij})$ is the matrix of initial endowments, and $p(w)$ denotes the Fisher market equilibrium prices given budget $w$, which is unique according to Fact \ref{fact}, ensuring that the mapping $p(w)$ is well-defined.

Let $T: \mathbb{R}^n \to \mathbb{R}^n$ denote the mapping $T(w) = E \cdot p(w)$. Then, the fixed-point equation \eqref{eq:fixed-point} characterizes the Arrow–Debreu market equilibrium as the solution to $w = T(w)$. A natural approach to solving this is through fixed-point iteration: starting from an initial budget $w^0$, we iteratively update the budget using
$$
w^{k+1} = T(w^k), \; k = 0,1,2, \cdots.
$$
At each iteration $k$, this requires computing the Fisher equilibrium prices for the given budget $w^k$, which can be efficiently obtained using Algorithm~\ref{algo:PDHCG}. The overall procedure for solving the Arrow–Debreu equilibrium is summarized in Algorithm~\ref{algo:arrow-debreu}.

\begin{algorithm}[t]
\DontPrintSemicolon
\SetAlgoLined
\SetInd{0.5em}{0.5em} 
\SetKwInput{KwInput}{Input}
\SetKwInput{KwOutput}{Output}
\SetKwInput{KwInitialize}{Initialize the inner loop}
\SetKwInput{KwRestart}{Restart the inner loop}
\SetKwFor{Repeat}{Repeat}{}{end} 
\SetKwFor{ForEach}{for each}{}{end} 

\KwInput{initial budget $w^0$ and tolerance $\varepsilon$.}
\ForEach{$k = 0, 1, 2, \ldots, $}{
Compute the Fisher equilibrium price $p(w^k)$ via Algorithm \ref{algo:PDHCG} \\
Update $w^{k+1} =  E \cdot p(w^k)$\\
Stop if $\|w^{k+1} - w^k\| \leq \varepsilon$
}
\KwOutput{$p(w^k)$}
\caption{Fixed Point Iteration}
\label{algo:arrow-debreu}
\end{algorithm}

To ensure the convergence of the fixed-point iteration method, a sufficient condition, as established by the Banach fixed-point theorem, is that the mapping $T$ is a contraction. We formally define this property below:
\begin{definition}
Let $\Omega \subseteq \mathbb{R}^n$ be a region. A mapping $T: \Omega \to \Omega$ is called a contractive mapping if there exists a constant $\gamma \in [0, 1)$ such that for all $w, w' \in \Omega$,
\begin{equation}
\|T(w) - T(w')\| \leq \gamma \|w - w'\|.
\end{equation}
\end{definition}
Under this condition, the sequence $\{w^k\}$ generated by $w^{k+1} = T(w^k)$ converges linearly to the unique fixed point $w^\star \in S$ satisfying $T(w^\star) = w^\star$.

However, the mapping $T(w) = E \cdot p(w)$ is not contractive in general. In fact, a two-dimensional counterexample provided in \cite{ye2008path} demonstrates that the fixed-point iteration for computing Arrow–Debreu equilibria can diverge. Fortunately, in large-scale settings, the contractiveness of the mapping $T$ can be ensured by incorporating randomness into both the utility matrix $U = (u_{ij})$ and the initial endowment matrix $E = (e_{ij})$. To formalize this, we introduce the following assumptions.

\begin{assumption}\label{ass:random-utility}
The utility coefficients $\{u_{ij}\}$ are i.i.d. random variables supported on $[0, 1]$, and there exists a constant $u_0 > 0$ such that $P(u_{ij} \geq u_0) = p > 0$.
\end{assumption}

\begin{assumption}\label{ass:random-endowment}
The initial endowments $\{e_{ij}\}$ are assumed to be i.i.d. random variables supported on $[0, 1]$, with mean $1/n$ and variance $\delta^2$. Moreover, there exists a constant $e_0 > 0$ such that $P(e_{ij} \geq e_0) = q > 0$.
\end{assumption}

Assumption~\ref{ass:random-utility} reflects the setting where players’ preferences for goods are independent and randomly distributed—a plausible model in scenarios like e-commerce platforms where user interests are highly diverse. Assumption~\ref{ass:random-endowment} further assumes that each player holds a random positive fraction of every good, which models decentralized and diverse initial resource distributions. 

Based on the randomness assumptions above, we can establish the following result, which ensures that $T(\Omega) \subseteq \Omega$ for some compact $\Omega$ which is away from the origin.
\begin{lemma}\label{lemma:boundedness}
    Let $\Delta = \{w\in \mathbb{R}^n \mid \sum_{i \in [n]} w_i = 1\}$ be the $n$-dimensional simplex and define 
    \begin{equation}\label{eq:omega}
        \Omega = \left\{w \in \Delta \mid w_i \geq \frac{u_0 e_0 q}{4n}, \, \forall i \in [n] \right\}.
    \end{equation}
    Under Assumptions~\ref{ass:random-utility} and~\ref{ass:random-endowment}, with probability at least $1 - m \exp(-2np^2) -n^2 \exp(-mp^2q^2 / 2) $, for any $w \in \Omega$, we have $T(w) \in \Omega$.
\end{lemma}
\begin{proof}
First of all, according to part (2) in Fact \ref{fact} and the fact that $\sum_{i \in [n]} e_{ij} = 1, \forall j \in [m]$, we have
$$
 \bm{1}_n^T T(w) = \bm{1}_n^T E\cdot p(w) = \bm{1}_m^T p(w) = \bm{1}_n^T w.
$$
Thus, if $w \in \Delta$, we know that $T(w) \in \Delta$. Next, given the budget $w$, the KKT condition \eqref{KKT-system-3} shows that the Fisher equilibrium price $p(w)$ satisfies
\begin{equation}\label{ineq:pw}
    p_j(w) \geq u_{ij} y_i = u_{ij}\frac{w_i}{t_i}  \geq \frac{u_{ij} w_i}{m}, \, \forall i \in [n],
\end{equation}
where the first equality holds because $t_i y_i =w_i$, and the second inequality holds since $t_i = \sum_{j \in m} u_{ij} x_{ij} \leq m$.
For any fixed $j \in [m]$, denoted by $\Phi_j = \{i \in [n] \mid u_{ij} \geq u_0\}$. Assumption \ref{ass:random-utility} tells us that $\{\bm{1}(i \in \Phi_j)\}$ are i.i.d. Bernoulli variables.  According to Hoeffding's inequality, we have 
$$
\mathbb{P}\left(|\Phi_j| =\sum_{i \in [n]} \bm{1}(i \in \Phi_j) \leq 2np \right) \leq \exp(-2np^2), \, \forall j \in [m].
$$
The union bound theory shows that, with probability at least $1 - m \exp(-2np^2)$, we have $|\Phi_j| \leq 2np,\, \forall j \in [m]$. Then, summing \eqref{ineq:pw} over all $i \in \Phi_j$, we obtain
$$
p_j(w) \geq \frac{1}{|\Phi_j|} \sum_{i \in \Phi_j} \frac{u_{ij} w_i}{m} \geq \frac{u_0}{2mnp} \sum_{i \in \Phi_j} w_i, \, \forall j \in [m].
$$
Next, for any $i, l \in [n]$, define $\Psi_l = \{j \in [m] \mid e_{lj} \geq e_0\}$ and $\Xi_{li} = \{ j \in [m] \mid e_{lj} \geq e_0, \, u_{ij} \geq u_0\}$. Similarly, we can show that, with probability at least $1 - n^2 \exp(-mp^2q^2 / 2)$, we have $|\Xi_{li}| \geq mpq / 2$. Then, we have
$$
[T(w)]_l = \sum_{j \in [m]} e_{lj} p_j(w) \geq e_0 \sum_{j \in \Psi_l } p_j(w) \geq \frac{u_0 e_0}{2mnp} \sum_{j \in \Psi_l } \sum_{i \in \Phi_j} w_i = \frac{u_0 e_0}{2mnp} \sum_{i \in [n]}  \sum_{j \in \Xi_{li} }  w_i \geq \frac{u_0 e_0 q}{4n}.
$$  
This complete the proof.
\end{proof}

According to Theorem 4 in \cite{ye2008path}, the Fisher equilibrium mapping $p(w)$ is continuous on the relative interior of the simplex $\Delta$. Since  $\Omega $ defined in \eqref{eq:omega} lies entirely within this relative interior of $\Delta$ and is compact, the continuity of $p(w)$ implies, by the Heine–Cantor theorem, that $p(w)$ is Lipschitz continuous on $\Omega$. This leads to the following corollary:
\begin{corollary}
    There exists some constant $\kappa >0$ such that $p(w)$ is $\kappa$-Lipschitz continuous on $\Omega$, i,e.,
    for all $w, w' \in \Omega$,
$$
\|p(w) - p(w')\| \leq \kappa \|w - w'\|.
$$
\end{corollary}
Finally, equipped with the $\kappa$-Lipschitz continuity of the mapping $p(w)$, we can establish that the operator $T$ is contractive under the assumption that the variance of the endowment variables $\{e_{ij}\}$ is sufficiently small. Consequently, the fixed-point iteration algorithm described in Algorithm~\ref{algo:arrow-debreu} converges linearly with high probability.
\begin{theorem}\label{thm:arrow-debreu}
    Under Assumption \ref{ass:random-utility} and \ref{ass:random-endowment}, and assume $\delta \leq \frac{\gamma}{2\kappa\sqrt{(m+n) \log (m+n)}}$ for some $\gamma \in (0, 1)$, then with probability at least $1 -m \exp(-2np^2) -n^2 \exp(-mp^2q^2 / 2) - c/(m+n)$, $T$ is a contractive mapping on $\Omega$ with parameter $\gamma$. Furthermore, the sequence $\{w^k\}$ generated by Algorithm \ref{algo:arrow-debreu} satisfies
    $$
    \|w^k - w^\star\| \leq \frac{\gamma^k}{1 - \gamma }\|w^0 - w^\star\|
    $$
    where $w^\star$ is the unique fixed point of the mapping $T$ on $\Delta$.
\end{theorem} 
\begin{proof}
For any $w\in \Omega$, we have
\begin{equation*}
    T(w) = E\cdot p(w) = \left( E - \frac{1}{n} \bm{1}_n \cdot \bm{1}_m^T \right) p(w) + \frac{1}{n} \bm{1}_n \cdot \bm{1}_m^T \cdot p(w) = \left( E - \frac{1}{n} \bm{1}_n \cdot \bm{1}_m^T \right) p(w) + \frac{1}{n} \bm{1}_n,
\end{equation*}
where the last equality hold because fo \eqref{eq:sum-pj}. Note that the expectation of the endowment matrix satisfies
$$
\mathbb{E} [E] = \frac{1}{n} \bm{1}_n \cdot \bm{1}_m^T \quad \text{and} \quad \mathbb{E}\left[\left(E - \frac{1}{n} \bm{1}_n \cdot \bm{1}_m^T \right) \left(E - \frac{1}{n} \bm{1}_n \cdot \bm{1}_m^T \right)^T\right] = m\delta^2 \bm{I}_n.
$$
Then, according to the matrix Bernstein inequality \cite[Theorem 1.6.2]{bernstein}, with probability at least $1 - c/(m+n)$ for some constant $c >0$, we have
    \begin{equation}\label{ineq:bernstein}
        \left\| E - \frac{1}{n} \bm{1}_n \cdot \bm{1}_m^T \right\| \leq 2\delta\sqrt{(m + n)\log (m+n)} .
    \end{equation}
Thus, for any $w, w' \in \Omega$, we have
\begin{equation*}
    \left\| T(w) - T(w')\right\| = \left\| \left( E - \frac{1}{n} \bm{1}_n \cdot \bm{1}_m^T \right) (p(w) - p(w') \right\| \leq 2\kappa\delta\sqrt{(m + n)\log (m+n)}  \|w -w'\| \leq \gamma \|w -w'\|.
\end{equation*}
The remaining part follows from the standard analysis of the Banach fixed-point theorem, and the probability is the union bound of the probability stated in Lemma \eqref{lemma:boundedness} and the probability with which \eqref{ineq:bernstein} holds.
\end{proof}

Theorem~\ref{thm:arrow-debreu} establishes that, under randomized assumptions on both the utility matrix $U$ and the endowment matrix $E$, the fixed-point iteration algorithm (Algorithm~\ref{algo:arrow-debreu}) converges linearly to the Arrow–Debreu market equilibrium. This result provides a practical and scalable approach for solving large-scale Arrow–Debreu models, for which no existing methods are known to be effective at such scale.

\section{Numerical Experiments}\label{sec:numerical}
In this section, we present numerical experiments on both the Fisher market and Arrow-Debreu market models to evaluate the performance of our proposed PDHCG algorithm. We compare it against several state-of-the-art methods using both synthetic and real-world datasets. The implementation code is available on GitHub\footnote{The code is available on \href{https://github.com/Huangyc98/PDHCG-in-Fisher.git}{https://github.com/Huangyc98/PDHCG-in-Fisher.git}}. The section is organized as follows. In Section~\ref{subsec:synthetic}, we evaluate the scalability and efficiency of our algorithm on synthetic Fisher market instances with varying problem sizes, ranging from $10^3$ to $10^7$. In Section~\ref{subsec:real-data}, we test our method on a real-world large-scale dataset--one-month transaction data from JD.com--to demonstrate practical effectiveness. In Section~\ref{subsec:arrow-debreu}, we turn to the Arrow-Debreu model and assess the proposed fixed-point iteration algorithm on synthetic datasets of different sizes, again from $10^3$ to $10^7$, to verify its efficiency and scalability.

\subsubsection{Solvers.} We compare the GPU implementation of our proposed PDHCG algorithm against the standard PDHG algorithm (with GPU support) and several state-of-the-art methods, including MOSEK--a leading commercial solver for exponential cone problems; PDCS \cite{lin2025pdcs}--a recent GPU-accelerated solver for large-scale conic programming; and PGLS \cite{gao2020first}—a first-order method specifically designed for the Fisher market model. For both MOSEK and PDCS, we reformulate the Fisher market problem into the following exponential cone programming form:
\begin{equation}
    \begin{split}
        \max_{x} \quad & \sum_{i \in [n]} w_i v_i \\
        \text{s.t.} \quad & \sum_{i \in [n]} x_{ij} = 1, \, \forall j \in [m] \\
        &  e^{v_i} \leq \sum_{j \in [m]} u_{ij} x_{ij}, \, \forall i \in [n] \\
        & x_{ij} \geq 0, \, \forall i,j.
    \end{split}
\end{equation}
\subsubsection{Criteria.}
We use the same criteria as those of PDLP \cite{applegate2021practical}, PDQP \cite{lu2023practical}, and PDHCG \cite{huang2024restarted}. The relative KKT error, which is defined as the maximal relative primal residual, dual residual, and primal-dual gap, is utilized as the progress metric for different algorithm. The relative primal residual, dual residual and primal dual gap for \eqref{eq:primal-dual-formulation} are defined as follows:
\[
r_{\text{primal}} = \frac{\max\left\{ \|\bm{1}_n^T x - 1\|_{\infty}, \,
  \max_{i\in[n]}\{|t_i - u_i^T x_i|\}
\right\}
}{1+ \max\{ \|\bm{1}_n^T x\|_{\infty}, \, \max_{i\in[n]}\{|t_i - u_i^T x_i|\} , \, 1\}},
\]
\[
r_{\text{dual}} =\frac{\max\left\{ 
 \|w / t - y\|_{\infty}\}, \,  \max_{j \in [m]} \{(p_j - \max_{i \in [n]}\{u_{ij} y_i\})_- \} \right\}}{1+ \max\left\{\|w /t\|_{\infty} , \,  \|y\|_{\infty}, \, \max_{j \in [m]} \{p_j - \max_{i \in [n]}\{u_{ij} y_i\} \} \right\} },
\]
\[
r_{\text{gap}} = \frac{ \max_{i ,j} \{x_{ij} (p_j - u_{ij} y_i)_+
\}}{1+ \max\left\{\|x\|_{\infty}, \,  \max_{i ,j} \{(p_j - u_{ij} y_i)_+
\} \right\}},
\]
where the term $w/t$ denotes entry-wise division between two vectors. For PDHCG w.r.t. \eqref{pdhcg:primal-dual-formulation}, we only need to set $t_i = u_i^T x_i, \forall i \in [n]$ and $y = w /t$. Note that if the relative KKT error is $0$, i.e., the relative primal residual, dual residual, and primal-dual gap are all equal to 0, then by KKT optimality conditions, we have found an optimal solution. Throughout this paper, we set the stopping criteria as the relative KKT error less than $10^{-4}$.

\subsubsection{Algorithmic Enhancements.} In both PDHG and PDHCG, we utilize several algorithmic enhancements upon Algorithm \ref{algo:PDHG} and Algorithm \ref{algo:PDHCG} to improve its practical
performance. More specifically,
\begin{itemize}
    \item \textbf{Preprocessing.} In the Fisher market model, scaling a consumer’s utility function by a positive constant does not affect their optimal choice. Therefore, at the beginning of the computation, we normalize each consumer's utility such that all the utilities range from (0,1).
    \item \textbf{Adaptive restart.} We mainly utilize the restarting strategy outlined in PDHCG \cite{huang2024restarted}, with parameter settings fine-tuned to address the specifics of our problem:
\[\beta_{\text{sufficient}} = 0.2, \, \beta_{\text{necessary}} = 0.8, \text{ and } \beta_{\text{artificial}} = 0.2.
\]
\item \textbf{Adaptive step-size.} The primal and the dual step-sizes of both PDHG and PDHCG are reparameterized as
$$
\tau = \eta / \omega, \text{ and } \sigma = \eta \omega
$$
where $\eta$, called step-size, controls the scale of the step-sizes, and $\omega$, called primal weight, 
balances the primal and the dual progress. For both step-size and primal weight adjustment, we impose the same strategy as in PDHCG \cite{huang2024restarted} with the parameter $\theta = 0.2$. Besides the update strategy, We also set an upper and low bound for the step-size and primal weight.The suitable range for step-size is $0.01\eta_{initial} \leq \eta \leq 3 \eta_{initial}$. For primal weight, We will check it every 3 restarts. If it exceeds the predefined boundary, we reset it to the initial primal weight.
\end{itemize}

\subsubsection{Computing environment}
All numerical experiments were executed on a high-performance computing cluster. We utilize an NVIDIA H100 GPU with 80GB of HBM3 memory and running CUDA version 12.4 for GPU based solver and an Intel Xeon Platinum 8469C operating at 2.60 GHz, equipped with 512 GB of RAM for CPU based solver. All the experiments are executed using Julia version 1.11.1.

\subsection{Fisher Market Equilibrium: Synthetic Data}\label{subsec:synthetic}
To evaluate the computational efficiency of the proposed algorithm at different problem scales, we randomly generate instances of varying sizes. In alignment with real-world markets—where the number of buyers typically far exceeds the number of goods, and each buyer is interested in only a small subset of items—our synthetic problem settings reflect these characteristics: the number of buyers $n$ is larger than the number of goods $m$, and the utility matrix $U$ is sparse. This setting allows us to test the proposed algorithm's scalability and robustness in sparse, buyer-dominant market environments. 

Specifically, each buyer’s budget is independently sampled from a uniform distribution, $w_i \sim \mathcal{U}(0, 1)$. The utility of buyer $i$ for good $j$, denoted $u_{ij}$, is set to zero with probability $1 - q$, and with probability $q$, it is independently drawn from a uniform distribution, $u_{ij} \sim \mathcal{U}(0, 1)$. The parameter $q \in (0,1]$ controls the sparsity level of the utility matrix.

\begin{table}[htbp]
\centering
\caption{Solving time (seconds) of different algorithms on synchetic dataset.}
\label{tab:random-fisher}
\renewcommand{\arraystretch}{1.5}
\begin{tabular}{p{1cm}p{1cm}cccccccc}
\toprule
\multirow{2}{*}{$n$} & \multirow{2}{*}{$m$} & \multirow{2}{*}{Sparsity} & \multicolumn{2}{c}{PDHCG} & \multicolumn{2}{c}{PDHG} & \multirow{2}{*}{PDCS} & \multirow{2}{*}{Mosek} & \multirow{2}{*}{PDLS}\\
\cmidrule(lr){4-5} \cmidrule(lr){6-7}
& & & Iteration & Time &  Iteration & Time & & &\\
\midrule
$10^3$  & 400   & 0.2   & 2,612 & 1.7  & 4303.8 & 4.5  &13.4 &\textbf{0.6} & 356.9 \\
    \hline
    $10^4$ & 4,000  & 0.02  & 2,758 & \textbf{3.6} &7604 & 7.1  & 1079.3 & 118.8 & f \\
    \hline
    $10^5$ & 4,000  & 0.02  & 1,200 & \textbf{8.3} & 10463 & 29.9 & f & 1038.9 & f \\
    \hline
    $10^6$ & 4,000  & 0.02  & 2,283 & \textbf{116.3} & 26793 & 739.1 &   f    &  f & f\\
    \hline
    $10^7$ & 4,000  & 0.01  & 2,423 & \textbf{589.7} & 29749 & 4,374.2 &  f     &  f & f\\
\bottomrule
\end{tabular}

\vspace{1mm}
{\footnotesize
\textit{Note.} “f” indicate failure due to exceeding the time limit (7200s) or returning errors.}
\end{table}

Table \ref{tab:random-fisher} reports the geometric mean of solving times across 10 randomly generated Fisher market instances with varying dimensions and sparsity levels. For small-scale problems (e.g., $m = 10^3$, $n = 400$), MOSEK performs exceptionally well, achieving the shortest runtime of 0.6 seconds. However, as the problem size increases, its performance deteriorates rapidly due to the high computational cost of matrix factorization, and it fails to solve instances with $m \geq 10^5$ within the 7200-second time limit. In contrast, the proposed PDHCG algorithm demonstrates strong scalability, solving the largest instance with $m = 10^7$ in 589.7 seconds. Although it requires more iterations than interior-point methods, it achieves competitive runtimes through matrix-free updates and GPU acceleration. Compared to the classical PDHG algorithm, PDHCG consistently uses fewer iterations and significantly less time, indicating faster convergence. Other first-order methods such as PDCS and PDLS fail to converge or exceed the time limit on medium to large-scale problems, underscoring PDHCG’s superior robustness and efficiency.

\subsection{Fisher Market Equilibrium: Real Data}\label{subsec:real-data}
JD.com, China’s largest retailer, provides transaction-level data \cite{shen2024jd} to MSOM members for data-driven research as part of the 2020 MSOM Data-Driven Research Challenge. The dataset records user activity from March 2018, covering over 2.5 million customers and more than 30,000 products on the platform. It includes both purchase histories and detailed browsing behaviors for individual users across a diverse product range. In our experiments, we use the number of product clicks as a proxy for users' utility functions, thereby capturing their underlying preferences. Each user’s budget $w_i$ is independently drawn from a uniform distribution, $w_i \sim \mathcal{U}(0, 1)$.

Table \ref{tab:fisher} reports the solving times of different algorithms on the JD.com transaction dataset, where $m$, $n$, and nnz represent the numbers of products, users, and nonzero entries in the utility matrix, respectively. We evaluate various problem sizes by selecting the most active users along with the products they are interested in. As shown in the table, PDHCG perform well in all instances. As the problem size increases, the performance advantage of PDHCG over Mosek and PDCS becomes increasingly evident. On large-scale instances, Mosek's CPU memory and PDCS's GPU memory requirements become prohibitively high, making it infeasible to run on our hardware. In contrast, PDHCG successfully solves the full-scale problem in 1222.9 seconds. Moreover, compared to the PDHG algorithm, PDHCG requires only about 15\% of the iterations and approximately 30\% of the solving time.

\begin{table}[htbp]
\centering
\caption{Solving time (seconds) of different algorithms on the JD.com transaction dataset.}
\label{tab:fisher}
\renewcommand{\arraystretch}{1.2}
\begin{tabular}{p{1.5cm}p{1.5cm}p{1.5cm}cccccc}
\toprule
\multirow{2}{*}{$n$} & \multirow{2}{*}{$m$} & \multirow{2}{*}{nnz} & \multicolumn{2}{c}{PDHCG} & \multicolumn{2}{c}{PDHG} & \multirow{2}{*}{PDCS} & \multirow{2}{*}{Mosek}\\
\cmidrule(lr){4-5} \cmidrule(lr){6-7}
& & & Iteration & Time &  Iteration & Time & & \\
\midrule
 2000  & 13541 & 136877 & 3930  &\textbf{7.0}  & 18299  & 19.8  & 62.2  & 25.1  \\
\hline
5000  & 15887 & 265220 & 4839  & \textbf{8.4} & 20805  & 23.6  & 256.3  & 78.2  \\
 \hline
20000 & 19449 & 698579 & 5894  & \textbf{12.4} & 39896  & 51.1  & f     & 526.1  \\
\hline
50000 & 21410 & 1278391 & 6865  & \textbf{17.1} & 266792  & 379.6  & f     & f \\
\hline
200000 & 23853 & 2937906 & 10090  & \textbf{41.3} & 345911  & 867.1  & f     & f \\
\hline
500000 & 25214 & 4679648 & 14899  & \textbf{85.8} & 545280  & 1623.3  & f     & f \\
\hline
1000000 & 26115 & 6189526 & 48726  & \textbf{313.9} & 706925  & 2301.2  & f     & f \\
\hline
2557841 & 27330 & 8073837 & 152373  & \textbf{1222.9} & 798329  & 2839.6  & f     & f \\
\bottomrule
\end{tabular}

\vspace{1mm}
{\footnotesize
\textit{Note.} “f” indicate failure due to exceeding the time limit (7200s) or returning errors.}
\end{table}

\subsection{Arrow-Debreu Market Equilibrium}\label{subsec:arrow-debreu}
To evaluate the computational efficiency of the proposed fixed-point algorithm for solving the Arrow–Debreu market equilibrium problem, we follow a similar setup to that in Section \ref{subsec:synthetic} by randomly generating instances of varying sizes. The utility matrix is generated in the same manner as described in Section \ref{subsec:synthetic}, while the initial endowment matrix $E \in \mathbb{R}^{n \times m}$ is constructed using the same approach but with a different sparsity level. This allows us to simulate realistic market environments where utility and endowment structures differ in sparsity.

\begin{table}[htbp]
\centering
\caption{Iterations and runtime (seconds) of fixed-point method for solving Arrow-Debreu Equilibrium.}
\label{tab:arrow-debreu}
\renewcommand{\arraystretch}{1.2}
\begin{tabular}{p{1cm}p{1cm}cccc}
\toprule
{$n$} & {$m$} & Sparsity of $E$ \;   & Sparsity of $U$ \; & Iterations & Time\\
\midrule
$10^3$ & 400   & 0.5   & 0.2   & 12.0  & 2.7  \\
\hline
    $10^4$ & 4000  & 0.05  & 0.02  & 9.3  & 13.0  \\ \hline
    $10^5$ & 4000  & 0.05  & 0.02  & 9.2  & 62.4  \\ \hline
    $10^6$ & 4000  & 0.05  & 0.02  & 9.7  & 560.7  \\ \hline
    $10^7$ & 4000  & 0.02  & 0.01  & 33.1  & 2305.3  \\
\bottomrule
\end{tabular}
\end{table}

Table \ref{tab:arrow-debreu} reports the number of iterations and runtime (in seconds) of the proposed fixed-point method for solving the Arrow–Debreu market equilibrium problem on randomly generated instances of increasing size. As shown in the table, the algorithm exhibits stable convergence behavior across a wide range of problem sizes. However, as the problem size reaches $n = 10^7$, both the number of iterations and the total runtime increase significantly, with runtime exceeding 2300 seconds. This can be attributed to the extremely low sparsity (i.e., dense interactions) and the increased cost of computing market excess demand over a large population. Despite this, the algorithm remains tractable and successfully solves the problem, highlighting its potential for large-scale applications in market analysis and economic modeling.

\section{Conclusion}\label{sec:conclusion}
In this work, we propose a GPU-accelerated primal-dual hybrid conjugate gradient (PDHCG) algorithm tailored for solving large-scale Fisher market equilibrium problems. By leveraging problem structure and incorporating efficient subproblem solvers such as bisection and 32-section search, our method achieves significantly improved convergence compared to standard PDHG. We establish a linear convergence rate for the proposed algorithm by proving a quadratic growth condition on the smoothed duality gap. Furthermore, we extend the approach to the Arrow–Debreu setting via a fixed-point iteration framework, and provide theoretical guarantees under mild assumptions. Extensive experiments on both synthetic and real-world datasets demonstrate the scalability and practical efficiency of our approach, outperforming several state-of-the-art solvers. This work offers a promising direction for high-performance computation in large-scale economic modeling and algorithmic market design.

\bibliographystyle{splncs04}
\bibliography{ref}

\end{document}